\documentclass[12pt]{amsart}

\newtheorem{theorem}{Theorem}[section]

\newtheorem{corollary}[theorem]{Corollary}

\newtheorem{lemma}[theorem]{Lemma}
\newtheorem{proposition}[theorem]{Proposition}
\newtheorem{conjecture}[theorem]{Conjecture}
\newtheorem{question}[theorem]{Question}

\theoremstyle{definition}
\newtheorem{definition}[theorem]{Definition}
\newtheorem{example}[theorem]{Example}

\theoremstyle{remark}
\newtheorem{remark}[theorem]{Remark}

\numberwithin{equation}{section}

\def\no{\nonumber}
\newcommand{\R}{\mathbb{R}}

\newcommand{\p}{\partial}

\newcommand{\sig}{\Sigma}

\newcommand{\n}{\nabla}

\newcommand{\la}{\langle}
\newcommand{\ra}{\rangle}

\newcommand{\fder}[2]{\frac{\partial #1}{\partial #2}}

\begin{document}
\title[Minimal surfaces and eigenvalue problems]
{Minimal surfaces and eigenvalue problems}
\author{Ailana Fraser}
\address{Department of Mathematics \\
                 University of British Columbia \\
                 Vancouver, BC V6T 1Z2}
\email{afraser@math.ubc.ca}
\author{Richard Schoen}
\address{Department of Mathematics \\
                 Stanford University \\
                 Stanford, CA 94305}
\thanks{The first author was partially supported by  the 
Natural Sciences and Engineering Research Council of Canada (NSERC) and the second
author was partially supported by NSF grant DMS-1105323}
\email{schoen@math.stanford.edu}

\begin{abstract} 
We show that metrics that maximize the $k$-th Steklov eigenvalue on surfaces with boundary arise from free boundary minimal surfaces in the unit ball. We prove several properties of the volumes of these minimal submanifolds. 
For free boundary minimal submanifolds in the ball 
we show that the boundary volume is reduced up to second order under conformal transformations of the ball. For two-dimensional stationary integer multiplicity rectifiable varifolds that are stationary for deformations that preserve the ball, we prove that the boundary length is reduced under conformal transformations.
We also give an overview of some of the known results on extremal metrics of the Laplacian on closed surfaces, and give a survey of our recent results from \cite{FS2} on extremal metrics for Steklov eigenvalues on compact surfaces with boundary. 
\end{abstract}

\maketitle

\section{Introduction}

Given a smooth compact surface $M$, the choice of a Riemannian metric $g$ on $M$ gives a
Laplace operator $\Delta_g$, which has a discrete set of eigenvalues 
\[
     \lambda_0=0<\lambda_1\leq \lambda_2\leq\ldots \leq \lambda_k \leq \ldots \rightarrow \infty.
\]
A basic question is: 
Assuming we fix the area to be $1$, what is the metric that maximizes the first eigenvalue? 
Does such a metric exist? 
If so what can we say about its geometry?
If we assume that we have a smooth metric $g$ that realizes the maximum, then it
turns out that the multiplicity of the eigenvalue is always at least $3$, and the
maximizing condition implies that there are independent eigenfunctions $u_1,\ldots, u_{n+1}$
with the property that $\sum |u_i|^2=1$ on $M$ and the map
$u=(u_1,\ldots,u_{n+1})$ defines a conformal map to $S^n$ with $n\geq 2$ (see \cite{N}).
This implies that the image surface $\Sigma=u(M)$  is a minimal surface in $S^n$;
that is, the mean curvature of $\Sigma$ is zero. Furthermore, the optimal metric $g$ is a positive
constant times the induced metric on $\Sigma$ from $S^n$. 

There are few surfaces for which maximizing metrics are known to exist.
In principle the cases with $\chi(M)\geq 0$ are understood: 

\vspace{2mm}

$\bullet$
For $S^2$ the constant curvature metric is the unique maximum by a result of Hersch \cite{H} from 1970. 

\vspace{2mm}

$\bullet$
For $RP^2$ the constant curvature metric is the unique maximum by a result of
Li and Yau \cite{LY} from the 1980s. The Veronese minimal embedding of $RP^2$ into 
$S^4$ is key. 

\vspace{2mm}

$\bullet$
For $T^2$ the flat metric on the $60^0$ rhombic torus is the unique maximum by a result of Nadirashvili \cite{N} from 1996. It can be minimally embedded into $S^5$ by first eigenfunctions.  
In 2000 El Soufi and Ilias \cite{EI2} showed that the only other smooth critical metric is the flat square torus which can be minimally embedded into $S^3$ as the Clifford torus. 

\vspace{2mm}

$\bullet$
For the Klein bottle the extremal metric is smooth and unique but not flat. This follows from work of Nadirashvili \cite{N} from 1996 on existence of a maximizer, Jakobson, Nadirashvili, and Polterovich \cite{JNP} from 2006 who constructed the metric, and El Soufi, Giacomini, and Jazar \cite{EGJ} from 2006 who proved it is unique. The metric arises on a minimal immersion of the Klein bottle into $S^4$. 

\vspace{2mm}

The case of the torus and the Klein bottle rely on a difficult existence theorem which was posed along with an outlined proof by Nadirashvili \cite{N}.

For large genus one might hope to understand the asymptotic behavior. If we fix
a surface $M$ of genus $\gamma$, then we can define 
\[ 
      \lambda^*(\gamma)=\sup\{\lambda_1(g)A(g):\ g \ \mbox{smooth metric on } M\},
\]
where $A(g)$ denotes the area of $(M,g)$.
Yang and Yau \cite{YY} have shown 
\[ 
      \lambda^*(\gamma)\leq 8\pi \left[\frac{\gamma+3}{2}\right].
\]
It was also shown by the combination of Buser, Berger, and Dodzuik \cite{BBD} and Brooks and Makover \cite{BM} that for large $\gamma$ there is a hyperbolic metric with
$\lambda_1\geq \frac{3}{16}$. This implies the lower bound
\[ 
    \lambda^*(\gamma)\geq \frac{3}{4}\pi(\gamma-1).
\]

It is natural to ask similar questions and look for optimal metrics on surfaces with boundary.
Note that a minimal submanifold $\Sigma^k$ in $S^n$ is naturally the boundary of a minimal
submanifold of the ball, the cone $C(\Sigma)$ over $\Sigma$. 
The coordinate functions of $\mathbb R^{n+1}$ restricted to $C(\Sigma)$ are harmonic functions which are homogeneous of degree $1$, so on the boundary they satisfy $\nabla_\eta x_i=x_i$ where $\eta$ is the outward unit normal vector to $\p C(\Sigma)$. 
More generally, a proper minimal submanifold $\Sigma$ of the unit ball $B^n$ which is orthogonal
to the sphere at the boundary is called a {\it free boundary submanifold}. These are
characterized by the condition that the coordinate functions are Steklov eigenfunctions with eigenvalue $1$; that is, $\Delta x_i=0$ on $\sig$ and $\nabla_\eta x_i=x_i$ on $\p \sig$. 
It turns out that surfaces of this type arise as eigenvalue extremals.

Some examples of free boundary submanifolds in the unit ball are the following:

\begin{example}
Equatorial $k$-planes $D^k \subset B^n$ are the simplest examples of free boundary minimal submanifolds. By a result of Nitsche \cite{Ni} any simply connected free boundary minimal surface in $B^3$ must be a flat equatorial disk. However, if we admit minimal surfaces of a different topological type, there are other examples, as we explain.
\end{example}

\begin{example}
Critical catenoid. Consider the catenoid parametrized on $\mathbb{R} \times S^1$ given by
\[
     \varphi(t,\theta)=(\cosh t \cos \theta, \cosh t \sin \theta, t ).
\]     
For a unique choice of $T_0$ the restriction of $\varphi$ to $[-T_0,T_0]\times S^1$ defines
a minimal embedding into a ball meeting the boundary of the ball orthogonally. 
We may rescale the radius of the ball to $1$ to get the {\it critical catenoid}. 
Explicitly $T_0$ is the unique positive solution of $t=\coth t$.
\end{example}

\begin{example}
Critical M\"obius band. We think of the M\"obius band $M$ as $\mathbb R\times S^1$ with the identification
$(t,\theta)\approx (-t,\theta+\pi)$. There is a minimal embedding of $M$ into $\mathbb R^4$
given by
\[ 
       \varphi(t,\theta)
       =(2\sinh t\cos\theta,2\sinh t\sin\theta,\cosh 2t\cos2\theta,\cosh 2t\sin2\theta).
\]
For a unique choice of $T_0$ the restriction of $\varphi$ to $[-T_0,T_0]\times S^1$ defines
a minimal embedding into a ball meeting the boundary of the ball orthogonally. 
We may rescale the radius of the ball to $1$ to get the {\it critical M\"obius band}. 
Explicitly $T_0$ is the unique positive solution of $\coth t=2\tanh 2t$. See \cite{FS2} for details.
\end{example}

A consequence of the results of \cite{FS2} (see Section \ref{section:survey} for a description) is that for every $k\geq 1$ there exists an embedded free boundary minimal surface in $B^3$ of genus $0$ with $k$ boundary components.
We expect that there are arbitrarily high genus free boundary solutions with three boundary components in $B^3$ which converge to the union of the critical catenoid and a disk through the origin orthogonal to the axis.

As stated above, free boundary submanifolds are characterized by the condition that the coordinate functions are Steklov eigenfunctions with eigenvalue $1$ (\cite{FS1} Lemma 2.2). If $(M,g)$ is a compact Riemannian manifold with boundary, the Steklov eigenvalue problem is:
\[  
       \left\{ \begin{array}{lll}
               \Delta_g u & = \; 0   & \mbox{ on } \; M \\
               \fder{u}{\eta} & = \; \sigma u  & \mbox{ on } \; \p M,
     \end{array} \right.
\]
where $\eta$ is the outward unit normal vector to $\p M$, $\sigma \in \mathbb{R}$, and $u \in C^\infty(M)$. Steklov eigenvalues are eigenvalues of the Dirichlet-to-Neumann map, which sends a given smooth function on the boundary to the normal derivative of its harmonic extension to the interior. The Dirichlet-to-Neumann map is a non-negative, self-adjoint operator with discrete spectrum
\[
     \sigma_0=0<\sigma_1\leq \sigma_2\leq\ldots \leq \sigma_k \leq \ldots \rightarrow \infty.
\]
The first  nonzero Steklov eigenvalue can be characterized variationally as
\[
      \sigma_1=\inf_{\int_{\p M} u=0}
       \frac{\int_M |\n u|^2 \,dv_M}{\int_{\p M} u^2 \,dv_{\p M}}
\]
and in general,
\[
        \sigma_k= \inf \left\{\frac{\int_M |\n u|^2 dv_M}{\int_{\p M} u^2 \, dv_{\p M}} \; : \; 
        \int_{\p M} u \phi_j=0 \mbox{ for } j=0, 1, 2, \ldots, k-1 \; \right\}
\]      
where $\phi_j$ is an eigenfunction corresponding to the eigenvalue $\sigma_j$, for $j=1, \ldots, k-1$.
The Steklov eigenvalues of surfaces satisfy a natural conformal invariance:
\begin{definition}
We say that two surfaces $M_1$ and $M_2$ are {\em $\sigma$-isometric} (resp. {\em $\sigma$-homothetic}) if there is a conformal diffeomorphism $F:M_1\to M_2$ that is an isometry (resp. a homothety) on the boundary; that is, $F^*g_2=\lambda^2 g_1$ and $\lambda=1$ on $\p M_1$ (resp. $\lambda=c$ on $\p M_1$ for some positive constant $c$). 
\end{definition}
Note that,
\begin{itemize}
\item If $M_1$ and $M_2$ are $\sigma$-isometric then they have the same Steklov
eigenvalues.
\item If $M_1$ and $M_2$ are $\sigma$-homothetic then they have the same normalized
Steklov eigenvalues; that is, $\sigma_i(M_1)L(\p M_1)=\sigma_i(M_2)L(\p M_2)$ for each $i$,
where $L(\p M_i)$ denotes the length of $\p M_i$.
\end{itemize}

The outline of the paper is as follows. In Section \ref{section:structure} we discuss the connection between maximizing metrics and minimal surfaces. The main result we prove is that if $g_0$ is a metric that maximizes the $k$-th normalized Steklov eigenvalue $\sigma_k(g)L_g(\p M)$, over all smooth metrics $g$ on a compact surface $M$ with boundary, then $g_0$ is $\sigma$-homothetic to the induced metric on a free boundary minimal surface in $B^n$. This generalizes the case $k=1$ which we proved in \cite{FS2}. If $g_0$ maximizes $\sigma_kL$ in a conformal class of metrics on $M$, we show that $(M,g_0)$ admits a harmonic map into $B^n$ satisfying the free boundary condition, and such that the Hopf differential is real on the boundary. We also give new proofs of the corresponding results for closed surfaces, which do not require analytic approximation. In Section \ref{section:volume} we discuss properties of the volumes of free boundary minimal submanifolds in the ball. We show that free boundary minimal surfaces in the ball maximize their boundary length in their conformal orbit, and for higher dimensional  free boundary minimal submanifolds we show that the boundary volume is reduced up to second order under conformal transformations of the ball. Finally, in Section \ref{section:survey} we give a brief overview of our recent results from \cite{FS2} on existence of maximizing metrics and sharp bounds for the first Steklov eigenvalue.

\section{The structure of extremal metrics} \label{section:structure}

In this section we summarize the connection between extremal metrics and minimal surfaces. Nadirashvili \cite{N} and El Soufi and Ilias \cite{EI2} showed that a metric that maximizes the functional $\lambda_1(g)$ on the set of Riemannian metrics $g$ of fixed area on a compact surface $M$ arises as the induced metric on a minimal surface in a sphere $S^n$. In \cite{EI3}, they extended the result to higher eigenvalues. Below we give a different proof, which follows some of the same arguments, but does not require analytic approximation.

\begin{proposition} \label{prop:extremalclosed}
Let $M$ be a compact surface without boundary, and suppose $g_0$ is a metric on $M$ such that 
\[
      \lambda_k(g_0)A_{g_0}(M)=\max_g \lambda_k(g) A_g(M)
\] 
where the max is over all smooth metrics on $M$. Then there exist independent $k$-th eigenfunctions $u_1, \ldots, u_n$ which, after rescaling the metric, give an isometric minimal immersion $u=(u_1, \ldots, u_n)$ of $M$ into the unit sphere $S^{n-1}$.
\end{proposition}

\begin{proof}
First we note that if $l < k$ and $\lambda_l(g_0)=\lambda_k(g_0)$, then $g_0$ maximizes $\lambda_l A$. Let $l$ be such that $\lambda_l(g_0)=\lambda_k(g_0)$ and $\lambda_{l-1}(g_0) < \lambda_l(g_0)$. 
Let $g(t)$ be a family of smooth metrics on $M$ with $g(0)=g_0$ and $\frac{d}{dt}g(t) =h(t)$, where  
$h(t)\in S^2(M)$ is a smooth family of symmetric $(0,2)$-tensor fields on $M$.
Denote by $E_i(g(t))$ the eigenspace corresponding to the $i$-th eigenvalue $\lambda_i(t)$ of $(M,g(t))$.
Define a quadratic form $Q_{h}$ on smooth functions $u$ on $M$ by
\[
      Q_{h}(u)=-\int_M \la \tau(u) + \frac{\lambda_l}{2} u^2g(t), h \ra \; da_t,
\]
where $\tau(u)$ is the stress-energy tensor of $u$ with respect to the metric $g(t)$,
\[
    \tau(u)= du \otimes du -\frac{1}{2} |\n u|^2 g.
\]

\begin{lemma} \label{lemma:deriv-closed}
$\lambda_l(t)$ is a Lipschitz function of $t$, and if $\dot{\lambda}_l(t_0)$ exists, then
\[
       \dot{\lambda}_l(t_0) = Q_{h}(u)
\] 
for any $u \in E_l(g(t_0))$ with $||u||_{L^2}=1$.
\end{lemma}
\begin{proof}
Let 
\[
      \mathcal{E}(t)=\cup_{i=0}^{l-1} E_i(g(t)).
\]
Then $\mathcal{E}(t)$ is of constant dimension $l$ and varies smoothly in $t$ for small $t$. Let
\[
      P_t: L^2(M,g(t))  \rightarrow \mathcal{E}(t)
\]
be the orthogonal projection onto $\mathcal{E}(t)$. 

To see that $\lambda_l(t)$ is Lipschitz for small t, let $t_1\neq t_2$ and assume without
loss of generality that $\lambda_l(t_1)\leq\lambda_l(t_2)$. Now let $u$ be an $l$-th 
eigenfunction for $g(t_1)$ normalized so that $\int_{M}u^2\ da_{t_1}=1$. It then follows easily from the fact that the path $g(t)$ is smooth that
\[ 
      |\int_M|\nabla^{t_1}u|^2\ da_{t_1}-\int_M|\nabla^{t_2}u|^2\ da_{t_2}|\leq c|t_1-t_2|
\]
and
\[ 
     |\int_{M}u^2\ da_{t_1}-\int_{M}u^2\ da_{t_2}|
      \leq c|t_1-t_2|,\ \ |P_{t_1}(u)-P_{t_2}(u)|\leq c|t_1-t_2|.
\]
Therefore we have
\[ 
          |\lambda_l(t_1)-\lambda_l(t_2)|
          =\lambda_l(t_2)-\lambda_l(t_1)
          \leq \frac{\int_M|\nabla^{t_2}(u-P_{t_2}(u))|^2\ da_{t_2}}
                         {\int_{M}(u-P_{t_2}(u))^2\ da_{t_2}}
                      -\int_M|\nabla^{t_1}u|^2\ da_{t_1}\leq c|t_1-t_2|
\]
and $\lambda_l(t)$ is Lipschitz.

Choose $u_0 \in E_l(g(t_0))$ and let  $u(t)=u_0-P_t(u_0)$.
Let
\[
       F(t)=\int_M |\n u(t)|^2 \; da_t -\lambda_l(t) \int_{M} u^2(t) \; da_t.
\]
Then $F(t) \geq 0$, and $F(t_0)=0$, and we have $\dot{F}(t_0)=0$.
Differentiating $F$ with respect to $t$ at $t=t_0$ we therefore obtain
\begin{align*}
    \int_M [\, 2 & \la \n u_0 , \n \dot{u}_0 \ra-\la  du_0 \otimes du_0 
    -\frac{1}{2} |\n u_0|^2 g , h \ra \,] \; da_{t_0} \\
    & = \dot{\lambda_l}(t_0) \int_{M}  u_0^2 \; da_{t_0} 
    + \lambda_l(t_0) \int_{M} [\,2u_0  \dot{u}_0 + \frac{1}{2}u_0^2 \la g, h \ra \,] \; da_{t_0}.
\end{align*}
Since $u_0$ is an $l$-th eigenfunction, we have
\[
     \int_M \la \n u_0, \n \dot{u}_0 \ra \; da_{t_0}  = \lambda_l(t_0) \int_{M} u_0 \, \dot u_0    \; da_{t_0}.    
\]
Using this, and if we normalize $u_0$ so that $||u_0||_{L^2}=1$, we have
\[
      \dot{\lambda}_l(t_0)=-\int_M \la \tau(u) + \frac{\lambda_l}{2} u^2g(t), h \ra \; da_{t_0} 
      = Q_{h}(u_0).
\]
\end{proof}

Let $L^2(S^2(M))$ denote the space of $L^2$ symmetric $(0,2)$-tensor fields on $M$ with respect to the metric $g_0$. 

\begin{lemma} \label{lemma:indefiniteclosed}
For any $\omega \in L^2(S^2(M))$ with $\int_{M} \la g_0, \omega \ra \;da_{g_0}=0$, there exists $u \in E_k(g_0)$ with $||u||_{L^2}=1$ such that $Q_\omega(u)=0$.
\end{lemma}
\begin{proof}
Let $\omega \in L^2(S^2(M))$, and assume that $\int_{M} \la g_0, \omega \ra \; da_{g_0}=0$. 
Since $C^\infty(S^2(M))$ is dense in $L^2(S^2(M))$, there is a sequence $h_i$ in $C^2(S^2(M))$ with $\int_{M} \la g_0, h_i \ra \; da_{g_0}=0$, such that $h_i \rightarrow \omega$ in $L^2$.

Let $g(t)=\frac{A_{g_0}(M)}{A_{g_0+th_i}(M)} (g_0+ th_i)$. 
Then $g(0)=g_0$, $A_{g(t)}(M)=A_{g_0}(M)$, and since
\[
      \frac{d}{dt}\Big|_{t=0} A_{g_0+th_i}(M)=\int_{M} \la g_0,  h_i \ra \; da_{g_0}=0
\]
we have $\frac{dg}{dt}\big|_{t=0} =h_i$.
Given any $\varepsilon >0$, by the fundamental theorem of calculus, 
\[
     \int_{-\varepsilon}^0 \dot{\lambda}_l(t) \; dt = \lambda_l(0) - \lambda_l(-\varepsilon) \geq 0
\]
by the assumption on $g_0$. Therefore there exists $t$, $-\varepsilon < t < 0$, such that 
$\dot{\lambda}_l(t)$ exists and $\dot{\lambda}_l(t) \geq 0$. Let $t_j$ be a sequence of points with $t_j<0$ and $t_j \rightarrow 0$, such that $\dot{\lambda}_l(t_j) \geq 0$. Choose $u_j \in E_l(g(t_j))=E_k(g(t_j))$ with $||u_j||_{L^2}=1$. Then, after passing to a subsequence, $u_j$ converges in $C^2(M)$ to an eigenfunction $u^{(i)}_- \in E_k(g_0)$ with $||u^{(i)}_-||_{L^2}=1$.
Since $Q_{h_i}(u_j)=\dot{\lambda}_l(t_j) \geq 0$, it follows that $Q_{h_i}(u^{(i)}_-) \geq 0$. By a similar argument, taking a limit from the right, there exists $u^{(i)}_+ \in E_k(g_0)$ with $||u^{(i)}_+||_{L^2}=1$, such that $Q_{h_i}(u^{(i)}_+) \leq 0$.

After passing to subsequences, $u^{(i)}_+ \rightarrow u_+$ and $u^{(i)}_- \rightarrow u_-$ in $C^2(M)$, and
\begin{align*}
      Q_\omega(u_+) &= \lim_{i \rightarrow \infty} Q_{h_i}(u^{(i)}_+)\leq 0 \\
      Q_\omega(u_-)  &= \lim_{i \rightarrow \infty} Q_{h_i}(u^{(i)}_-)\geq 0.
\end{align*}
\end{proof}

Without loss of generality, rescale the metric $g_0$ so that $\lambda_k(g_0)=2$.
The remainder of the argument is as in \cite{EI2}. Let $K$ be the convex hull in $L^2(S^2(M))$ of 
\[
        \{ \, \tau(u) + u^2 g_0 \, : \, u \in E_k(g_0) \}.
\]        
We claim that $g_0 \in K$.
If $g_0 \notin K$, then since $K$ is a convex cone which lies in a finite dimensional subspace, the Hahn-Banach theorem implies the existence of $\omega \in L^2(S^2(M))$ that separates $g_0$ from $K$; in particular such that
\begin{align*}
      \int_M \la g_0, \omega \ra \; da_{g_0} & >0, \quad \mbox{and} \\
      \int_M \la \tau(u)+u^2g_0, \omega \ra \; da_{g_0} & <0 
      \quad \mbox{ for all } u \in E_k(g_0) \setminus \{0\}.
\end{align*}
Let 
\[
       \tilde{\omega}
       =\omega-\left(\frac{1}{2A_{g_0}(M)}\int_M \la g_0, \omega \ra \; da_{g_0} \right)g_0.
\]
Then, $\int_M \la g_0, \tilde{\omega} \ra \;da_{g_0}=0$, and
\begin{align*}
      Q_{\tilde{\omega}}(u)
      &= - \int_M \la  \tau(u) + u^2 g_0, \tilde{\omega} \ra \; da_{g_0} \\
      &= -\int_M \la \tau(u) + u^2 g_0, \omega \ra \; da_{g_0}
            + \frac{\int_M \la g_0, \omega \ra \; da_{g_0}}{A_{g_0} (M)} \, \int_M u^2 \; da_{g_0} \\
      &> 0.
\end{align*}
This contradicts Lemma \ref{lemma:indefiniteclosed}. Therefore, $g_0 \in K$, and since $K$ is contained in a finite dimensional subspace, there exist independent eigenfunctions 
$u_1, \ldots, u_n \in E_k(g_0)$ such that
\[
     \sum_{i=1}^n (du_i \otimes du_i -\frac{1}{2} |\n u_i|^2 g_0 + u_i^2 g_0) =g_0
\]
This implies that
$\sum_i u_i^2 = 1$ and 
$\sum_i du_i \otimes du_i = g_0$.
Thus $u=(u_1, \ldots, u_n): M \rightarrow S^{n-1}$ is an isometric minimal immersion. 
\end{proof}

We turn now to the Steklov eigenvalue problem. In \cite{FS2} we proved that metrics that maximize the first Steklov eigenvalue arise from free boundary minimal surfaces in the unit ball. The proof is  more complicated in the Steklov eigenvalue setting because the quadratic form $Q$ has both an interior and a boundary term. Here we extend our result \cite{FS2} for $\sigma_1$ to higher Steklov eigenvalues.

\begin{proposition} \label{prop:extremal}
If $M$ is a surface with boundary, and $g_0$ is a metric on $M$ with
\[
      \sigma_k(g_0)L_{g_0}(\p M)=\max_g \sigma_k(g) L_g(\p M)
\] 
where the max is over all smooth metrics on $M$. Then there exist independent $k$-th eigenfunctions $u_1, \ldots, u_n$ which give a conformal minimal immersion 
$u=(u_1, \ldots, u_n)$ of $M$ into the unit ball $B^n$ such that u(M) is a free boundary solution, and up to rescaling of the metric $u$ is an isometry on $\p M$.
\end{proposition}

\begin{proof}
First we note that if $l < k$ and $\sigma_l(g_0)=\sigma_k(g_0)$, then $g_0$ maximizes $\sigma_l L$. Let $l$ be such that $\sigma_l(g_0)=\sigma_k(g_0)$ and $\sigma_{l-1}(g_0) < \sigma_l(g_0)$. 
Let $g(t)$ be a family of smooth metrics on $M$ with $g(0)=g_0$ and $\frac{d}{dt}g(t) =h(t)$, where  
$h(t)\in S^2(M)$ is a smooth family of symmetric $(0,2)$-tensor fields on $M$.
Denote by $E_i(g(t))$ the eigenspace corresponding to the $i$-th Steklov  eigenvalue $\sigma_i(t)$ of $(M,g(t))$.
Define a quadratic form $Q_{h}$ on smooth functions $u$ on $M$ as follows
\[
      Q_{h}(u)=-\int_M \la \tau(u) , h \ra \; da_t 
         -\frac{\sigma_l(t)}{2} \int_{\p M} u^2 h(T,T) \; ds_t,
\]
where $T$ is the unit tangent to $\p M$ for the metric $g_0$.
\begin{lemma}
$\sigma_l(t)$ is a Lipschitz function of $t$, and if $\dot{\sigma}_l(t_0)$ exists, then
\[
       \dot{\sigma}_l(t_0) = Q_{h}(u)
\] 
for any $u \in E_l(g(t_0))$ with $||u||_{L^2}=1$.
\end{lemma}
\begin{proof}
Let 
\[
      \mathcal{E}(t)=\cup_{i=0}^{l-1} E_i(g(t)).
\]
Then $\mathcal{E}(t)$ is of constant dimension $l$ and varies smoothly in $t$ for small $t$. Let
\[
      P_t: L^2(\p M,g(t))  \rightarrow \mathcal{E}(t)
\]
be the orthogonal projection onto $\mathcal{E}(t)$. 

To see that $\sigma_l(t)$ is Lipschitz for small t, let $t_1\neq t_2$ and assume without
loss of generality that $\sigma_l(t_1)\leq\sigma_l(t_2)$. Now let $u$ be an $l$-th Steklov
eigenfunction for $g(t_1)$ normalized so that $\int_{\partial M}u^2\ ds_{t_1}=1$. It then follows easily from the fact that the path $g(t)$ is smooth that
\[ 
      |\int_M|\nabla^{t_1}u|^2\ da_{t_1}-\int_M|\nabla^{t_2}u|^2\ da_{t_2}|\leq c|t_1-t_2|
\]
and
\[ 
     |\int_{\partial M}u^2\ ds_{t_1}-\int_{\partial M}u^2\ ds_{t_2}|
      \leq c|t_1-t_2|,\ \ |P_{t_1}(u)-P_{t_2}(u)|\leq c|t_1-t_2|.
\]
Therefore we have
\[ 
          |\sigma_l(t_1)-\sigma_l(t_2)|
          =\sigma_l(t_2)-\sigma_l(t_1)
          \leq \frac{\int_M|\nabla^{t_2}(u-P_{t_2}(u))|^2\ da_{t_2}}
                         {\int_{\partial M}(u-P_{t_2}(u))^2\ ds_{t_2}}
                      -\int_M|\nabla^{t_1}u|^2\ da_{t_1}\leq c|t_1-t_2|
\]
and $\sigma_l(t)$ is Lipschitz.

Choose $u_0 \in E_l(g(t_0))$ and let  $u(t)=u_0-P_t(u_0)$.
Let
\[
       F(t)=\int_M |\n u(t)|^2 \; da_t -\sigma_l(t) \int_{\p M} u^2(t) \; ds_t.
\]
Then $F(t) \geq 0$, and $F(t_0)=0$, and we have $\dot{F}(t_0)=0$.
Differentiating $F$ with respect to $t$ at $t=t_0$ we therefore obtain
\begin{align*}
    \int_M [\, 2 & \la \n u_0 , \n \dot{u}_0 \ra-\la  du_0 \otimes du_0 
    -\frac{1}{2} |\n u_0|^2 g , h \ra \,] \; da_{t_0} \\
    & = \dot{\sigma_l}(t_0) \int_{\p M}  u_0^2 \; ds_{t_0} 
    + \sigma_l(t_0) \int_{\p M} [\,2u_0  \dot{u}_0 + \frac{1}{2}u_0^2 h(T,T) \,] \; ds_{t_0}.
\end{align*}
Since $u_0$ is an $l$-th Steklov eigenfunction, we have
\[
     \int_M \la \n u_0, \n \dot{u}_0 \ra \; da_{t_0}  = \sigma_l(t_0) \int_{\p M} u_0 \, \dot u_0    \; ds_{t_0}.    
\]
Using this, and if we normalize $u_0$ so that $||u_0||_{L^2}=1$, we have
\[
      \dot{\sigma}_l(t_0)=-\int_M \la du_0 \otimes du_0 -\frac{1}{2} |\n u_0|^2 g , h \ra \; da_{t_0} 
         -\frac{\sigma_l(t_0)}{2} \int_{\p M} u_0^2 h(T,T) \; ds_{t_0} = Q_{h}(u_0).
\]
\end{proof}

Consider the Hilbert space $\mathcal{H}=L^2(S^2(M)) \times L^2(\p M)$, the space of pairs of $L^2$ symmetric $(0,2)$-tensor fields on $M$ and $L^2$ functions on boundary of $M$ with respect to the metric $g_0$.

\begin{lemma} \label{lemma:indefinite}
For any $(\omega,f) \in \mathcal{H}$ with $\int_{\p M} f \;ds=0$, there exists $u \in E_k(g_0)$ with $||u||_{L^2}=1$ such that $\la (\omega,\frac{\sigma_k(g_0)}{2}f), (\tau(u),u^2) \ra_{L^2}=0$.
\end{lemma}
\begin{proof}
Let $(\omega,f) \in \mathcal{H}$, and assume that $\int_{\p M} f \; ds=0$. 
Since $C^\infty(S^2(M)) \times C^\infty(M)$ is dense in $L^2(S^2(M)) \times L^2(M))$, we can approximate $(\omega, f)$ arbitrarily closely in $L^2$ by a smooth pair $(h,\tilde{f})$ with $\int_{\p M} \tilde{f} \; ds =0$. We may redefine $h$ in a neighbourhood of the boundary to a smooth tensor whose restriction to $\p M$ is equal to the function $\tilde{f}$, and such that the change in the $L^2$ norm is arbitrarily small. In this way, we obtain a smooth sequence $h_i$ with $\int_{\p M} h_i(T,T) \; ds=0$, such that $(h_i, h_i(T,T)) \rightarrow (\omega,f)$ in $L^2$.

Let $g(t)=\frac{L_{g_0}(\p M)}{L_{g_0+th_i}(\p M)} (g_0+ th_i)$. 
Then $g(0)=g_0$, $L_{g(t)}(\p M)=L_{g_0}(\p M)$, and since
\[
      \frac{d}{dt}\Big|_{t=0} L_{g_0+th_i}(\p M)=\int_{\p M} h_i(T,T) \; ds=0
\]
we have $\frac{dg}{dt}\big|_{t=0} =h_i$.
Given any $\varepsilon >0$, by the fundamental theorem of calculus, 
\[
     \int_{-\varepsilon}^0 \dot{\sigma}_l(t) \; dt = \sigma_l(0) - \sigma_l(-\varepsilon) \geq 0
\]
by the assumption on $g_0$. Therefore there exists $t$, $-\varepsilon < t < 0$, such that 
$\dot{\sigma}_l(t)$ exists and $\dot{\sigma}_l(t) \geq 0$. Let $t_j$ be a sequence of points with $t_j<0$ and $t_j \rightarrow 0$, such that $\dot{\sigma}_l(t_j) \geq 0$. Choose $u_j \in E_l(g(t_j))=E_k(g(t_j))$ with $||u_j||_{L^2}=1$. Then, after passing to a subsequence, $u_j$ converges in $C^2(M)$ to an eigenfunction $u^{(i)}_- \in E_k(g_0)$ with $||u^{(i)}_-||_{L^2}=1$.
Since $Q_{h_i}(u_j)=\dot{\sigma}_l(t_j) \geq 0$, it follows that $Q_{h_i}(u^{(i)}_-) \geq 0$. By a similar argument, taking a limit from the right, there exists $u^{(i)}_+ \in E_k(g_0)$ with $||u^{(i)}_+||_{L^2}=1$, such that $Q_{h_i}(u^{(i)}_+) \leq 0$.

After passing to subsequences, $u^{(i)}_+ \rightarrow u_+$ and $u^{(i)}_- \rightarrow u_-$ in $C^2(M)$, and
\begin{align*}
      \la (\omega,\frac{\sigma_k(g_0)}{2}f), (\tau(u_+),u_+^2) \ra_{L^2} 
      &= -\lim_{i \rightarrow \infty} Q_{h_i}(u^{(i)}_+)\geq 0 \\
     \la (\omega,\frac{\sigma_k(g_0)}{2}f), (\tau(u_-),u_-^2) \ra_{L^2}  
     &= -\lim_{i \rightarrow \infty} Q_{h_i}(u^{(i)}_-)\leq 0.
\end{align*}
\end{proof}

Without loss of generality, rescale the metric $g_0$ so that $\sigma_k(g_0)=1$.
Let $K$ be the convex hull in $\mathcal{H}$ of 
\[
        \{ \, (\tau(u), u^2) \, : \, u \in E_k(g_0) \}.
\]        
We claim that $(0,1) \in K$.
If $(0,1) \notin K$, then since $K$ is a convex cone which lies in a finite dimensional subspace, the Hahn-Banach theorem implies the existence of $(\omega,f) \in \mathcal{H}$ such that
\begin{align*}
      \la (\omega,\frac{1}{2}f), (0,1) \ra_{L^2} & >0, \quad \mbox{and} \\
      \la (\omega,\frac{1}{2}f), (\tau(u),u^2) \ra_{L^2} & <0 \quad \mbox{ for all } u \in E_k(g_0) \setminus \{0\}.
\end{align*}
Let $\tilde{f}=f-\frac{1}{L_{g_0}(\p M)}\int_{\p M} f \; ds$. Then, $\int_{\p M} \tilde{f} \;ds=0$, and
\begin{align*}
      \la (\omega, \frac{1}{2}\tilde{f}), (\tau(u),u^2) \ra_{L^2}
      &= \int_M \la \omega, \tau(u) \ra \; da +\frac{1}{2} \int_{\p M} \tilde{f} u^2 \; ds \\
      &=  \int_M \la \omega, \tau(u) \ra \; da + \frac{1}{2}\int_{\p M} f u^2 \; ds 
            - \frac{\int_{\p M} f}{2L_{g_0} (\p M)} \, \int_{\p M} u^2 \; ds \\
      &= \la (\omega,\frac{1}{2}f), (\tau(u),u^2) \ra_{L^2} 
          -\frac{\int_{\p M} u^2 \; ds}{L_{g_0}(\p M)} \; \la (\omega,\frac{1}{2}f),(0,1) \ra_{L^2} \\
      &< 0.
\end{align*}
This contradicts Lemma \ref{lemma:indefinite}. Therefore, $(0,1) \in K$, and since $K$ is contained in a finite dimensional subspace, there exist independent eigenfunctions 
$u_1, \ldots, u_n \in E_k(g_0)$ such that
\begin{align*}
     0 & = \sum_{i=1}^n \tau(u_i) 
     = \sum_{i=1}^n (du_i \otimes du_i -\frac{1}{2} |\n u_i|^2 g_0) \quad \mbox{ on } M \\
     1& = \sum_{i=1}^n u_i^2  \quad \mbox{ on } \p M
\end{align*}
Thus $u=(u_1, \ldots, u_n): M \rightarrow B^n$ is a conformal minimal immersion. Since $u_i$ is a $k$-th Steklov eigenfunction and $\sigma_k(g_0)=1$, we have $\fder{u_i}{\eta}=u_i$ on $\p M$ for $i=1, \ldots, n$. Therefore, 
\[
      \left|\fder{u}{\eta}\right|^2= |u|^2=1 \qquad \mbox{on } \p M,
\]
and since $u$ is conformal, $u$ is an isometry on $\p M$.      
\end{proof}

We conclude this section by mentioning the corresponding results for the problem of maximizing eigenvalues within a conformal class of metrics. El Soufi and Ilias \cite{EI2} showed that if a metric $g$ maximizes $\lambda_k$ among all metrics in its conformal class, then $(M,g)$ admits a harmonic map of constant energy density into a sphere. 
Here we give an alternate proof of this using the method of our proof of Proposition \ref{prop:extremalclosed} above which does not require analytic approximation.

\begin{proposition}
Let $M$ be a compact surface without boundary, and suppose $g_0$ is a metric on $M$ such that 
\[
      \lambda_k(g_0)A_{g_0}(M)=\max_g \lambda_k(g) A_g(M)
\] 
where the max is over all smooth metrics on $M$ in the conformal class of $g_0$. Then there exist independent $k$-th eigenfunctions $u_1, \ldots, u_n$ such that $\sum_{i=1}^n u_i^2=1$. That is, $u=(u_1, \ldots, u_n): (M, g_0) \rightarrow S^{n-1}$ is a harmonic map with constant energy density $e(u)=\lambda_k/2$.
\end{proposition}

\begin{proof}
Let $g(t)=e^{\alpha(t)}g_0$ be a smooth family of metrics with $g(0)=g_0$ and $\frac{d}{dt}g(t)=\varphi(t) g_0$. Let $l$ be such that $\lambda_l(g_0)=\lambda_k(g_0)$ and $\lambda_{l-1}(g_0) < \lambda_l(g_0)$. Then, using Lemma \ref{lemma:deriv-closed}, we have that $\lambda_l(t)$ is a Lipschitz function of $t$, and if $\dot{\lambda}_l(t_0)$ exists, then
\[
      \dot{\lambda}_l(t_0)=Q_\varphi(u):=- \lambda_l(g_0) \int_M u^2 \varphi  \; da_{t_0}.
\]
By the same argument as in Lemma \ref{lemma:indefiniteclosed}, for any $f \in L^2(M)$ with $\int_M f \; da_{g_0}=0$, there exists $u \in E_k(g_0)$ with $\|u\|_{L^2}=1$ such that $Q_f(u)=0$.
We may then complete the proof using the Hahn-Banach argument as before. Let $K$ be the convex hull in $L^2(M)$ of $\{ u^2 \; : \; u \in E_k(g_0) \}$. Arguments similar to the proof of Proposition \ref{prop:extremalclosed} imply that the constant function $1$ belongs to $K$. Therefore, there exist $u_1, \ldots, u_n \in E_k(g_0)$ such that
$\sum_{i=1}^n u_i^2=1$. Then $u=(u_1, \ldots, u_n) : M \rightarrow S^2$ is a harmonic map, and $2e(u)u=\Delta u = \lambda_k u$, so $e(u)=\lambda_k/2$.
\end{proof}

For maximizing Steklov eigenvalues in a conformal class of metrics we have the following. The proof is slightly simpler than the case of maximizing $\sigma_kL$ over all metrics, because when restricted to conformal variations of the metric the quadratic form $Q$ does not have an interior term.

\begin{proposition} 
If $M$ is a surface with boundary, and $g_0$ is a metric on $M$ with
\[
      \sigma_k(g_0)L_{g_0}(\p M)=\max_g \sigma_k(g) L_g(\p M)
\] 
where the max is over all smooth metrics on $M$ in the conformal class of $g_0$. Then there exist independent $k$-th eigenfunctions $u_1, \ldots, u_n$ such that $\sum_{i=1}^n u_i^2=1$ on $\p M$. 
That is, $u= (u_1, \ldots, u_n): (M, \p M) \rightarrow (B^n, \p B^n)$ is a harmonic map and the Hopf differential of $u$ is real on $\p M$.
\end{proposition}

\begin{proof}
Let $g(t)=e^{\alpha(t)}g_0$ be a smooth family of metrics with $g(0)=g_0$ and $\frac{d}{dt}g(t)=\varphi(t) g_0$. Let $l$ be such that $\sigma_l(g_0)=\sigma_k(g_0)$ and $\sigma_{l-1}(g_0) < \sigma_l(g_0)$. Then, using Lemma \ref{lemma:deriv-closed}, we have that $\sigma_l(t)$ is a Lipschitz function of $t$, and if $\dot{\sigma}_l(t_0)$ exists, then
\[
      \dot{\sigma}_l(t_0)=Q_\varphi(u):=- \frac{\sigma_l(g_0)}{2} \int_{\p M} u^2 \varphi  \; ds_{t_0}.
\]
By the same argument as in Lemma \ref{lemma:indefiniteclosed}, for any $f \in L^2(\p M)$ with $\int_{\p M} f \; ds_{g_0}=0$, there exists $u \in E_k(g_0)$ with $\|u\|_{L^2(\p M)}=1$ such that $Q_f(u)=0$.
We may then complete the proof using the Hahn-Banach argument as before. Let $K$ be the convex hull in $L^2(\p M)$ of $\{ u^2 \; : \; u \in E_k(g_0) \}$. Arguments similar to those in the proof of Proposition \ref{prop:extremalclosed} imply that the constant function $1$ belongs to $K$. Therefore, there exist $u_1, \ldots, u_n \in E_k(g_0)$ such that
$\sum_{i=1}^n u_i^2=1$ on $\p M$. Then $u=(u_1, \ldots, u_n) : M \rightarrow B^n$ is a harmonic map, with $u(\p M) \subset B^n$ and meeting $\p B^n$ orthogonally. 
Then $u_\eta \cdot u_T=0$ on $\p M$,
and the Hopf differential of $u$ is real on the boundary of $M$.
\end{proof}

To elucidate the condition described in the previous proposition, we describe the notion of a
$\frac{1}{2}$-harmonic map from $\p M$ to $\p\Omega$ where $M$ is a Riemann surface with boundary and $\Omega$ is a Riemannian manifold with boundary. We say that 
$u:\p M\to\p \Omega$ is {\it $\frac{1}{2}$-harmonic} if there is a harmonic map 
$\hat{u}:M\to\Omega$ with $\hat{u}=u$ on $\p M$ whose energy is stationary with respect to variations which preserve $\Omega$ but do not necessarily preserve $\p \Omega$. This is the 
condition that the normal derivative of the map is parallel to the unit normal of $\p\Omega$ along the image of $u$. This notion has been formulated by Da Lio and Rivi\'ere \cite{DR}
in case $M$ is the unit disk and $\Omega$ is a domain in $\mathbb R^n$, and they have
obtained delicate regularity results for $\frac{1}{2}$-harmonic maps from the unit circle to
the sphere. In case $\Omega$ is a domain in $\mathbb R^n$, and $M$ is the unit disk, such maps arise as critical points for the $H^{1/2}$ energy of the boundary map among maps from $\p M$ to
$\p\Omega$, and that is the reason for the terminology. The conclusions of the proposition are 
equivalent to the condition that $u$ is a $\frac{1}{2}$-harmonic from $\p M$ to the sphere 
$\p B^n$. The results of \cite{DR} will be important for obtaining the regularity of metrics which maximize Steklov eigenvalues on surfaces with a fixed conformal structure.

\section{Volumes of Minimal Submanifolds} \label{section:volume}

In this section we discuss some properties of the volumes of the minimal surfaces that arise in the extremal eigenvalue problems as described in the previous section. In the case of closed surfaces, Li and Yau \cite{LY} proved that minimal surfaces in $S^n$ maximize their area in their conformal orbit, and El Soufi and Ilias \cite{EI1} generalized this result to higher dimensional minimal submanifolds in $S^n$:
\begin{theorem} \label{theorem:lyei}
Let $\sig$ be a compact $k$-dimensional minimal submanifold in $S^n$. Suppose $f \in G \setminus O(n+1)$, where $G$ is the group of conformal transformations of $S^n$. If $\sig$ is not isometric to $S^k$ then
\[
        |f(\sig)| < |\sig|.
\]
\end{theorem}
This result in the two-dimensional case plays a key role in the characterization of the maximizing metrics for $\lambda_1$ on the torus and Klein bottle (see \cite{MR}, \cite{N}, \cite{JNP}).
The corresponding question for free boundary minimal submanifolds in the ball remains open. 
\begin{question} \label{question:boundaryvolume}
Suppose $\sig$ is a $k$-dimensional free boundary minimal submanifold in $B^n$, and $f: B^n \rightarrow B^n$ is conformal. Is it true that $|f(\sig)| \leq | \sig|$?
\end{question}
On the other hand, for $k=2$, 
we have the following result on conformal images of {\em boundaries} of free boundary minimal surfaces. We proved this in \cite{FS1} (Theorem 5.3) using different methods.

\begin{theorem} \label{theorem:length}
Let $\Sigma$ be a minimal surface in $B^n$ with $\p \Sigma \subset \p B^n$, meeting $\ B^n$ orthogonally. 
Suppose $f: B^n \rightarrow B^n$ is conformal. Then
\[
      |f(\p\Sigma)| \leq |\p\Sigma|.
\]
\end{theorem}

\begin{proof} 
We use a first variation argument. Recall that the first variation formula is 
\[ 
        \int_\Sigma \mbox{div}_\Sigma(V)=\int_{\p\Sigma}V\cdot x\ ds
\]
where $\mbox{div}_\Sigma(V)=\sum_i \nabla_{e_i}V\cdot e_i$ in an orthonormal tangent basis.

Let $u$ be the length magnification factor of $f$; that is, $f^*(\delta)=u^2\delta$ where $\delta$
is the Euclidean metric. It can be checked that there is a $y\in{\mathbb R}^n$ with $|y|>1$
so that $u(x)=\frac{|y|^2-1}{|x-y|^2}$. If we make the choice 
\[ 
       V=\frac{x-y}{|x-y|^2}
\]
in the first variation we can check that $\mbox{div}_\Sigma V\geq 0$ and 
$V\cdot x=\frac{1}{2}(1-u)$ on $\p\Sigma$. It then follows that
\[ 
      0\leq\int_{\p\Sigma}(1-u)\ ds=|\p\Sigma|-|f(\p\Sigma)|.
\] 
\end{proof}

\begin{remark} 
The proof above uses only the first variation, and thus works
for integer multiplicity rectifiable varifolds which are stationary for deformations which preserve the
ball. This is the class of objects which are produced by Almgren \cite{A} in the min/max theory for
the variational problem.
\end{remark}

A lower bound on areas for free boundary surfaces follows from the boundary length decrease under conformal maps of the ball since such maps can be chosen which blow up any chosen point $x \in \p \Sigma$ and the limit is an equatorial disk with area $\pi$ (for details see Theorem 5.4 in \cite{FS1}).

\begin{theorem} 
Let $\Sigma^2$ be a minimal surface in $B^n$ with $\p\Sigma \subset \p B^n$,
meeting $\p B^n$ orthogonally. Then
\[
        |\Sigma| \geq \pi.
\]
Equivalently, $| \p\Sigma| \geq 2\pi$, since $|\p\Sigma|=2|\Sigma|$.
\end{theorem}
When $\p\Sigma$ is a single curve this result follows from work of Croke and Weinstein \cite{CW} on a sharp lower length bound on `balanced' curves.

\begin{corollary}
The sharp isoperimetric inequality holds for free boundary minimal surfaces in the ball:
\[
         |\sig|  \leq \frac{|\p \sig|^2}{4 \pi}.
\]
\end{corollary}  

It is natural to conjecture that 
Theorem \ref{theorem:length} holds 
for higher dimensional free boundary minimal submanifolds in the ball:

\begin{conjecture} 
If $\Sigma$ is a $k$ dimensional minimal submanifold in the ball $B^n$ with $\p\Sigma\subset S^{n-1}$ and with the conormal vector of $\Sigma$ equal to the position vector, then for any conformal transformation $f$ of the ball we have $|f(\p\Sigma)|\leq |\p\Sigma|$.
\end{conjecture}

Evidence for this conjecture is suggested by the following. By Theorem \ref{theorem:length} it is true for $k=2$. In the special case of a cone, Theorem \ref{theorem:lyei} implies the conjecture. Moreover, it is true in general that the boundary volume is reduced up to second order, as we prove in Theorem \ref{theorem:2nd} below.
Finally, a consequence of the conjecture is that $|\Sigma|\geq |B^k|$. This inequality was shown recently by a direct monotonicity-style argument by Brendle \cite{Br}.

\begin{theorem} \label{theorem:2nd}
Let $\Sigma$ be a $k$ dimensional minimal submanifold in the ball $B^n$ with $\p\Sigma\subset S^{n-1}$ and with the conormal vector of $\Sigma$ equal to the position vector. Given any $v \in S^{n-1}$, let $f_t$ be the one parameter family of conformal transformations of the ball generated by the gradient of the linear function in the direction $v$. Then,
\[
      \left.\frac{d^2}{dt^2}\right|_{t=0} |f_t(\p \sig)| 
      =-(k-1)k \int_{\p \sig} |v^\perp|^2 \; dv \;\leq \;0.
\]
\end{theorem}

\begin{proof}
Let $\Sigma$ be a $k$ dimensional minimal submanifold in the ball $B^n$ with $\p\Sigma\subset S^{n-1}$ and with the conormal vector of $\Sigma$ equal to the position vector. Given $v \in S^{n-1}$, let $\ell=x \cdot v$, where $x$ is the position vector, and let $X = \n^{S^{n-1}} \ell=v^{\top^{S^{n-1}}}$, the component of $v$ tangent to $S^{n-1}$. Then,
\begin{align*}
     \frac{d}{dt} |f_t(\p \sig)| &= \int_{\p \sig} \mbox{div}_{f_t(\p \sig)} X \; dv_t \\
                                             &= -(k-1) \int_{\p \sig} \ell \; dv_t
\end{align*}
and
\begin{align} \label{eq:2nd}
    \frac{d^2}{dt^2} |f_t(\p \sig)| &= -(k-1) \int_{\p\sig} \left[ X(\ell) - (k-1) \ell^2 \right] \; dv_t  \no\\
    &= -(k-1) \int_{\p\sig} \left[ |\nabla^{S^{n-1}} \ell |^2 - (k-1) \ell^2 \right] \; dv_t  \\ 
    &= -(k-1) \int_{\p\sig} \left[ (1-\ell^2) - (k-1) \ell^2 \right] \; dv_t \no\\ 
    &= -(k-1) \int_{\p \sig} (1-k \ell^2) \; dv_t, \no
\end{align}
where in the third equality we have used 
$|\n^{S^{n-1}} \ell |^2=| v^{\top^{S^{n-1}}} |^2=1 - (v\cdot x)^2= 1-\ell^2$.
Set $V=x-k\ell v$. Then $V \cdot x = 1 - k \ell^2$,
and by the first variation formula,
\[
     \int_\sig \mbox{div}_\sig V = \int_{\p \sig} V \cdot x = \int_{\p \sig} 1-k \ell^2.
\]
On the other hand,
\[
      \mbox{div}_\sig V = k - k |v^\top|^2 = k |v^\perp|^2,
\]
where $v^\top$ and $v^\perp$ denote the components of $v$ tangent and normal to $\sig$.
Therefore,
\[
    \int_{\p \sig} (1 - k \ell^2) \; dv = k \int_{\p \sig} |v^\perp|^2 \; dv.
\]
Substituting this in (\ref{eq:2nd}), we obtain
\[
    \left.\frac{d^2}{dt^2} |f_t(\p \sig)|\right|_{t=0} = -(k-1)k \int_{\p \sig} |v^\perp|^2 \; dv \;\leq \;0.
\]    
\end{proof}
This shows that the boundary volume of a free boundary minimal submanifold $\sig$ in the ball is reduced up to second order under conformal images, however the following weaker question than Question \ref{question:boundaryvolume}, of whether the volume of $\sig$ itself is reduced up to second order, remains open:\begin{question}
Suppose $\sig$ is a $k$-dimensional free boundary minimal submanifold in $B^n$, and $X$ is a conformal Killing vector field in $B^n$. Is it true that $\delta^2 \sig (X,X) \leq 0$?
\end{question}
On the other hand, by  a subtle argument, we were able to show that the second variation of volume of $\sig$ is nonpositive for the vector fields $v^\perp$, for any $v \in \mathbb{R}^n$ (\cite{FS2} Theorem 3.1).
\begin{theorem}
If $\sig^{k}$ is a free boundary minimal submanifold in $B^n$ and $v\in\R^n$, then
we have
\[ 
     \delta^2 \sig(v^\perp,v^\perp)=-k\int_{\sig} |v^\perp|^2\ dv.
\]
If $\sig$ is not contained in a product $\sig_0 \times \mathbb{R}$ where $\sig_0$ is a free boundary solution, then the Morse index of $\sig$ is at least $n$.
In particular, if $k=2$ and $\sig$ is not a plane disk, its index is at least $n$.

\end{theorem}

This index estimate for free boundary minimal submanifolds in the ball plays a key role in our sharp bounds for the first nonzero Steklov eigenvalue on the annulus and M\"obius band from \cite{FS2}, which we survey in the next section.

\section{Sharp Steklov eigenvalue bounds} \label{section:survey}
 
 In this section we give a brief overview of our recent results from \cite{FS2} on existence of extremal metrics and sharp bounds for the first Steklov eigenvalue. If we fix a surface $M$ of genus $\gamma$ with $k$ boundary components, we can define
\[
           \sigma^*(\gamma,k)=\sup_g \,\sigma_1(g)\,L_g(\p M)
\]
where the supremum is over all smooth metrics on $M$.
By a result of Weinstock \cite{W},  $\sigma^*(0,1)=2\pi$, and the supremum is achieved uniquely up to $\sigma$-homothety by the Euclidean disk $D$. 
In general, we have the coarse upper bound
\[
        \sigma^*(\gamma,k) \, \leq \, \min\{2 \pi ( \gamma +k), \  8\pi[(\gamma+3)/2]\},
\]
which is a combination \cite{FS1} Theorem 2.3 and \cite{KN} Proposition 1. A basic question is: What is the value of $\sigma^*(\gamma,k)$ for other surfaces? Does a maximizing metric exist?

In \cite{FS2} we show that for any compact surface $M$ with boundary, a smooth maximizing metric $g$ exists on $M$ provided the conformal structure is controlled for any metric near the maximum. For surfaces of genus zero with arbitrarily many boundary components we prove boundedness of the conformal structure for nearly maximizing metrics. Thus, we have the following existence result for maximizing metrics on surfaces of genus zero.
\begin{theorem} \label{theorem:existence}
For any $k\geq 1$ there exists a smooth metric $g$ on the surface of genus $0$ with $k$ boundary components with the property $\sigma_1(g)L_g(\p M)=\sigma^*(0,k)$.
\end{theorem}
In the case of the annulus and the M\"obius band, we explicitly characterize the maximizing metrics. 
Recall that for any free boundary minimal surface in the ball, the coordinate functions are Steklov eigenfunctions with eigenvalue 1. In the case of the critical catenoid and the critical M\"obius band, the coordinate functions are in fact {\em first} Steklov eigenfunctions (\cite{FS1} section 3, \cite{FS2} Proposition 7.1). Moreover, we show that these are the only free boundary annuli and M\"obius bands such that the coordinate functions are first eigenfunctions:
\begin{theorem}
Assume that $\Sigma$ is a free boundary minimal annulus in $B^n$ such that the coordinate functions are first eigenfunctions. Then $n=3$ and $\Sigma$ is the critical catenoid.
\end{theorem}
\begin{theorem}
Assume that $\Sigma$ is a free boundary minimal M\"obius band in $B^n$ such that the coordinate functions are first eigenfunctions. Then $n=4$ and $\Sigma$ is the critical M\"obius band.
\end{theorem}
By Theorem \ref{theorem:existence} there exists a maximizing metric on the annulus and on the M\"obius band. By Proposition \ref{prop:extremal} these maximizing metrics are $\sigma$-homothetic to the induced metrics from free boundary minimal immersions of the annulus and M\"obius band in $B^n$. But then by the minimal surface uniqueness results above, these immersions must be congruent to the critical catenoid and critical M\"obius band. Thus we have the following sharp eigenvalue bounds and explicit characterizations of maximizing metrics on the annulus and M\"obius band.
\begin{theorem}
For any metric annulus $M$ 
\[
         \sigma_1L\leq (\sigma_1L)_{cc}
\]
with equality if and only if $M$ is $\sigma$-homothetic to the critical catenoid. 
In particular,
\[
     \sigma^*(0,2)=(\sigma_1L)_{cc}\approx 4\pi/1.2.
\]
\end{theorem}
\begin{theorem}
For any metric M\"obius band $M$ 
\[
       \sigma_1L\leq (\sigma_1L)_{cmb}=2\pi\sqrt{3}
\]
with equality if and only if $M$ is $\sigma$-homothetic to the critical M\"obius band. 
\end{theorem}
For surfaces of genus 0 and $k \geq 3$ boundary components we are not able to explicitly characterize the maximizing metrics, but we show that the metrics arise from free boundary surfaces in $B^3$ which are embedded and star-shaped with respect to the origin, and we analyze the limit as $k$ goes to infinity.
\begin{theorem}
The sequence $\sigma^*(0,k)$ is strictly increasing in $k$ and converges to $4\pi$ as $k$ tends to infinity. For each $k$ a maximizing metric is achieved by a free boundary minimal surface $\Sigma_k$ in $B^3$ of area less than $2\pi$. The limit of these minimal surfaces as $k$ tends to infinity is a double disk.
\end{theorem}
As a consequence, we have the following existence theorem for free boundary minimal surfaces in the ball.
 \begin{corollary}
 For every $k\geq 1$ there is an embedded minimal surface in $B^3$ of genus $0$ with $k$ boundary components satisfying the free boundary condition. Moreover these surfaces are embedded by first eigenfunctions.
 \end{corollary}

\bibliographystyle{plain}

\end{document}